\documentclass[12pt,a4paper]{article}

\usepackage{latexsym}
\usepackage{graphicx}
\usepackage{subfigure}
\usepackage{amsmath}
\usepackage{float}
\usepackage{color}
\usepackage{cite}
\usepackage{enumerate}
\usepackage{epsfig}
\usepackage{xcolor}
\usepackage{caption}
\usepackage{changes}
\usepackage{amssymb}
\usepackage{booktabs}
\usepackage{caption}
\usepackage{tikz}
\usetikzlibrary{positioning,calc}
\usepackage[hidelinks]{hyperref}
\usetikzlibrary{arrows.meta}
\definecolor{interactionblue}{HTML}{216A9D}

\newcommand{\qed}{\hfill $\Box$}

\newtheorem{theorem}{Theorem}[section]
\newtheorem{lemma}[theorem]{Lemma}
\newtheorem{proposition}[theorem]{Proposition}
\newtheorem{corollary}[theorem]{Corollary}
\newtheorem{conjecture}[theorem]{Conjecture}

\newtheorem{example}[theorem]{Example}

\newcommand{\rn}{\chi'_{\mathrm{rn}}}
\newcommand{\scind}{\chi'_s}

\begin{document}

\baselineskip=17pt
\parindent=0.5cm

\title{\Large {\bf Linear bounds for rich-neighbor edge-colorings}}
\bigskip

\author{Jiangxu Kong \thanks{\footnotesize Corresponding author. Email: kjz@hznu.edu.cn.}\\
\normalsize School of Mathematics, Hangzhou Normal University, Hangzhou, 311121, China
}

\date{}
\maketitle

%%%%%%%%%%%%%%%%%%%%%%%%%%%%%%%%%%%%%%%%%%%%%%%%%%%%%%%%%%%%%%%%%%%%%%%
%
\begin{abstract}

\baselineskip=16 pt
An edge in a proper edge-coloring is rich if all edges adjacent to it receive pairwise distinct colors. A rich-neighbor edge-coloring is a proper edge-coloring in which every non-isolated edge is adjacent to a rich edge. The rich-neighbor chromatic index $\rn(G)$ is the minimum number of colors in such a coloring. We prove that every loopless multigraph $G$ satisfies $\rn(G)\le3\chi'(G)$, where $\chi'(G)$ is the chromatic index. Exact values are also determined for complete bipartite graphs and complete graphs. For all positive integers $m,n$, $\rn(K_{m,n})=m+n-1$, except that $\rn(K_{2,2})=4$. For $n\ge3$, $\rn(K_n)=2n-3$, except that $\rn(K_4)=6$ and $\rn(K_6)=10$. Moreover, for $n\ge7$, the rich edges in every optimal rich-neighbor edge-coloring of $K_n$ form a star forest.

\medskip


\end{abstract}
%
%%%%%%%%%%%%%%%%%%%%%%%%%%%%%%%%%%%%%%%%%%%%%%%%%%%%%%%%%%%%%%%%%%%%%%%

\baselineskip=17pt

\section{Introduction}\label{a-sec-introduction}
All graphs and multigraphs considered in this paper are finite, undirected, and loopless. Unless stated otherwise, graphs are simple. Two distinct edges are adjacent if they share an endpoint; in particular, parallel edges are adjacent.

A proper edge-coloring of a graph $G$ assigns distinct colors to adjacent edges. The chromatic index $\chi'(G)$ is the minimum number of colors in a proper edge-coloring of $G$. In a proper edge-coloring, an edge is \emph{rich} if the edges adjacent to it receive pairwise distinct colors. A \emph{rich-neighbor edge-coloring} is a proper edge-coloring in which
every non-isolated edge is adjacent to a rich edge. The \emph{rich-neighbor chromatic index} $\rn(G)$ is the minimum number of colors in such a coloring of $G$. In particular, every non-isolated rich edge must also be adjacent to a rich edge. 

A \emph{strong edge-coloring}, introduced by Fouquet and Jolivet~\cite{FouquetJolivet}, is a proper edge-coloring in which every edge is rich. The \emph{strong chromatic index} $\scind(G)$ is the minimum number of colors in a strong edge-coloring of $G$. Hence 
\begin{equation}\label{a-basic-comparison}
\chi'(G)\le \rn(G)\le \scind(G).
\end{equation}

Petru\v{s}evski and \v{S}krekovski~\cite{PetrusevskiSkrekovski} introduced rich-neighbor edge-coloring and proved that $\rn(G)\le7$ for every subcubic graph $G$. They conjectured that $\rn(G)\le5$ for every connected subcubic graph $G$ other than $K_4$. Chen, Tan, and Zhou~\cite{ChenTanZhou} verified this conjecture for graphs whose vertices of degree $3$ form an independent set. Chen, Nan, and Zhou~\cite{ChenNanZhou} obtained further bounds for sparse and claw-free subcubic graphs.

These results concern graphs of maximum degree at most $3$. For a graph $G$ with at least one edge, greedy coloring of the square of its line graph gives $\scind(G)\le2\Delta(G)^2-2\Delta(G)+1$. Together with~\eqref{a-basic-comparison}, this yields a quadratic upper bound on $\rn(G)$. Since $\scind(K_{\Delta,\Delta})=\Delta^2$, a quadratic dependence on the maximum degree is necessary for strong edge-coloring. It is therefore natural to ask whether there is an absolute constant $C$ such that $\rn(G)\le C\Delta(G)$ for every graph $G$.

Our first result gives an affirmative answer for multigraphs. An edge-coloring $\varphi$ \emph{refines} an edge-coloring $\psi$ of the same graph if every color class of $\varphi$ is contained in a color class of $\psi$.

\begin{theorem}\label{a-main-general}
Every proper $k$-edge-coloring of a multigraph $G$ can be refined to a rich-neighbor edge-coloring with at most $3k$
colors. In particular, $\rn(G)\le3\chi'(G)$.
\end{theorem}

Theorem~\ref{a-main-general} and Shannon's theorem~\cite{Shannon} give $\rn(G)\le3\lfloor3\Delta(G)/2\rfloor$ for every  multigraph $G$. For simple graphs, Vizing's theorem~\cite{Vizing} yields $\rn(G)\le3\Delta(G)+3$. If $\chi'(G)=\Delta(G)$, then $\rn(G)\le3\Delta(G)$. By K\H{o}nig's line-coloring theorem~\cite{Konig}, this bound holds for every bipartite multigraph.

For a connected  multigraph with at least three vertices, we construct a simple forest $F$ such that every edge has a neighbor in $F$. Each component of $F$ is a star with at least two edges, except possibly one component isomorphic to $P_4$. Given a proper $k$-edge-coloring, we define a conflict graph on each color class to record the constraints required to make the edges of $F$ rich. We show that each conflict graph is $2$-degenerate and hence $3$-colorable. These $3$-colorings split each original color class into at most three classes, yielding a refinement in which every edge of $F$ is rich. The resulting coloring is a rich-neighbor edge-coloring with at most $3k$ colors.

We next determine the rich-neighbor chromatic index of complete bipartite graphs and complete graphs.

\begin{theorem}\label{a-main-bipartite}
For all positive integers $m$ and $n$,
\(
\rn(K_{m,n})=m+n-1,
\)
except that $\rn(K_{2,2})=4$.
\end{theorem}
For each positive integer $\Delta$, define
\[
R(\Delta)=\max\{\rn(G)\colon\, G\text{ is a simple graph
and }\Delta(G)\le\Delta\}.
\]
Theorem~\ref{a-main-general} and Vizing's theorem~\cite{Vizing} give a uniform upper bound on the values in this set, so the maximum exists. Applying Theorem~\ref{a-main-bipartite} to $K_{\Delta,\Delta}$ gives the lower bound, and hence
\begin{equation}\label{a-extremal-range} 2\Delta-1\le R(\Delta)\le3\Delta+3
\qquad(\Delta\ge3).
\end{equation}
Thus $R(\Delta)=\Theta(\Delta)$. Moreover, since
$\chi'(K_{\Delta,\Delta})=\Delta$, any bound of the form
$\rn(G)\le c\Delta(G)+O(1)$ or $\rn(G)\le c\chi'(G)+O(1)$
that holds for all finite simple graphs must satisfy $c\ge2$.

For complete graphs, we obtain the following exact values.

\begin{theorem}\label{a-main-complete}
For every integer $n\ge3$,
\(
\rn(K_n)=
\begin{cases}
2n-2,&n\in\{4,6\},\\
2n-3,&n\notin\{4,6\}.
\end{cases}
\)
\end{theorem}

For a proper edge-coloring $\varphi$ of $G$, the \emph{rich-edge graph} $R_\varphi$ is the spanning subgraph of $G$
whose edges are precisely the rich edges under $\varphi$. A component is \emph{nontrivial} if it contains an edge.
A rich-neighbor edge-coloring of $G$ is \emph{optimal} if it uses exactly $\rn(G)$ colors. For $n\ge7$, let $s(n)$ denote the minimum number of nontrivial components of $R_\varphi$ over all optimal rich-neighbor edge-colorings $\varphi$ of $K_n$.

\begin{theorem}\label{a-main-structure}
For every integer $n\ge7$, the rich-edge graph of every optimal rich-neighbor edge-coloring of $K_n$ is a star forest with at most one isolated vertex and at least two edges in each nontrivial component. Moreover,
\[
s(n)=
\begin{cases}
1,&n\text{ odd},\\
2,&n\equiv0\pmod4,\\
3,&n\equiv2\pmod4.
\end{cases}
\]
\end{theorem}

%For multigraphs, we also study binary refinements, which split each original color class into at most two classes. We prove $\rn(G)\le2\chi'(G)$ when $G$ admits a total edge-dominating star forest with a balanced signed interaction multigraph~\cite{Harary,Zaslavsky}. Balance is sufficient but not necessary for binary refinement; a color-frustration parameter bounds the extra colors needed by this method when balance fails. We conjecture that $\rn(G)\le2\chi'(G)$ for every multigraph $G$.

%Section~\ref{a-sec-linear} introduces the notation and proves the general linear bound. Section~\ref{a-sec-exact} determines the exact values for complete bipartite graphs and complete graphs and establishes the structure of the rich edges in optimal colorings of the latter. Section~\ref{a-sec-questions} concludes with a conjecture.

\section{The linear bound}\label{a-sec-linear}\label{a-sec-notation}
For a multigraph $G$, let $d_G(v)$ denote the degree of $v$, and let $\delta(G)$ and $\Delta(G)$ denote the minimum and
maximum degrees, respectively. Let $N_G(v)$ denote the open neighborhood of $v$, and put $N_G[v]=N_G(v)\cup\{v\}$.
For $U\subseteq V(G)$, let $G[U]$ denote the subgraph induced by $U$. For a positive integer $k$, write $[k]=\{1,\ldots,k\}$. For a vertex set $X$, let $K_X$ denote the complete graph on $X$. For disjoint vertex sets $X$ and $Y$, let $K_{X,Y}$ denote the complete bipartite graph with parts $X$ and $Y$, and write $E(X,Y)=\{xy\colon\, x\in X,\ y\in Y\}$ for its edge set. For a vertex $v\notin X$, write $E(v,X)=E(\{v\},X)$. For an edge $e=uv$, put $V_G(e)=\{u,v\}$, and let $N_G(e)$ denote the set of edges adjacent to $e$.
For $A\subseteq E(G)$, put $V_G(A)=\bigcup_{e\in A}V_G(e)$. An edge $e$ is isolated if and only if $N_G(e)=\varnothing$.
In a proper edge-coloring $\varphi$, an edge $e$ is rich if and only if $\varphi$ is injective on $N_G(e)$.

A graph is \emph{$k$-degenerate} if every nonempty subgraph has a vertex of degree at most $k$. Every $k$-degenerate graph is $(k+1)$-colorable. Let $G$ be a multigraph. A subgraph $F\subseteq G$ is \emph{total edge-dominating} if every
edge of $G$ is adjacent to an edge of $F$. The total edge-dominating forests we construct are simple and have no isolated vertices. The \emph{support graph} of $G$ is the simple graph obtained by replacing each class of parallel edges with a single edge. A star is \emph{large} if it has at least two edges. An \emph{almost-star forest} is a forest whose components are
large stars, except possibly one component isomorphic to $P_4$, the path on four vertices.

For a proper edge-coloring $\varphi$ of a multigraph $G$ with color set $\mathcal C$, the \emph{palette} at a vertex $v$ is
$\Phi(v)=\{\varphi(e)\colon\, e\text{ is incident with }v\}$. For adjacent vertices $u$ and $v$, let $\mu_G(u,v)$ denote the
number of edges joining them. For any edge $e$ joining $u$ and $v$,
\begin{equation}\label{a-palette-rich-multigraph}
e\text{ is rich}
\quad\Longleftrightarrow\quad
|\Phi(u)\cap\Phi(v)|=\mu_G(u,v).
\end{equation}
Indeed, the $\mu_G(u,v)$ edges joining $u$ and $v$ receive distinct colors, all belonging to $\Phi(u)\cap\Phi(v)$.
Any additional common color occurs on two distinct edges adjacent to $e$. Conversely, if two distinct edges adjacent
to $e$ have the same color, then properness forces them to meet opposite endpoints of $e$, and their common color cannot
occur on an edge joining $u$ and $v$.

For an edge $uv$ of a simple graph, $\mu_G(u,v)=1$, and hence
\begin{equation}\label{a-palette-rich}
uv\text{ is rich}
\quad\Longleftrightarrow\quad
\Phi(u)\cap\Phi(v)=\{\varphi(uv)\}.
\end{equation}

\begin{lemma}\label{a-local-lower}
If an edge joining $u$ and $v$ is rich in a proper edge-coloring of a multigraph $G$, then the coloring uses at least $d_G(u)+d_G(v)-\mu_G(u,v)$ colors. For a simple graph, this bound is $d_G(u)+d_G(v)-1$. In particular, every nonempty simple graph $G$ with $\delta(G)\ge2$ satisfies $\rn(G)\ge2\delta(G)-1$.
\end{lemma}

\begin{proof}
Let $\varphi$ be the coloring, with color set $\mathcal C$. By properness and~\eqref{a-palette-rich-multigraph},
\[
|\mathcal C|
\ge |\Phi(u)\cup\Phi(v)|
= d_G(u)+d_G(v)-\mu_G(u,v).
\]
If $G$ is simple, then $\mu_G(u,v)=1$. If, in addition, $G$ is nonempty and $\delta(G)\ge2$, then every edge is non-isolated, so every rich-neighbor edge-coloring contains a rich edge. Applying the bound to such an edge gives $\rn(G)\ge2\delta(G)-1$. \qed
\end{proof}

\subsection{Total edge-dominating forests}\label{a-sec-forests}

\begin{lemma}\label{a-lem-cover}
Let $G$ be a multigraph, and let $F\subseteq G$ be a simple forest each of whose components contains at least two edges.
Then $F$ is total edge-dominating if and only if $V(G)\setminus V(F)$ is independent.
\end{lemma}
\begin{proof}
If $F$ is total edge-dominating, then every edge of $G$ has an endpoint in $V(F)$, so $V(G)\setminus V(F)$ is independent.
Conversely, suppose that $V(G)\setminus V(F)$ is independent. Every edge in $E(G)\setminus E(F)$ has an endpoint in $V(F)$ and hence is adjacent to an edge of $F$. Every edge of $F$ also has a neighbor in $F$, since its component contains at least two edges. \qed
\end{proof}

We use the following lemma of Boyer and Goddard~\cite[Lemma~13]{BoyerGoddard}.

\begin{lemma}\label{a-lem-star-removal}
Every connected simple graph $G$ of order at least three contains a large star $S$, not necessarily induced, such that
$G-V(S)$ has at most one nontrivial component.
\end{lemma}

\begin{theorem}\label{a-thm-almost-star}
Every connected multigraph $G$ of order at least three contains a simple almost-star forest $F$ such that $V(G)\setminus V(F)$ is independent. Consequently, $F$ is total edge-dominating.
\end{theorem}
\begin{proof}
A forest in the support graph of $G$ can be realized as an isomorphic simple forest in $G$ by choosing one edge from each
corresponding parallel class. Since $G$ and its support graph have the same independent sets, we may assume that $G$ is
simple. We induct on $|V(G)|$. If $|V(G)|=3$, take a spanning two-edge star.

For $|V(G)|\ge4$, Lemma~\ref{a-lem-star-removal} gives a large star $S\cong K_{1,t}$ such that $G-V(S)$ has at most one
nontrivial component. Let $c$ be its center and $x_1,\ldots,x_t$ its leaves. Let $I$ be the set of isolated vertices of $G-V(S)$, and let $H$ be its nontrivial component, if one exists. If $H$ does not exist, take $F=S$. If $|V(H)|\ge3$, let $F_H$ be given by induction and take $F=S\cup F_H$. In both cases, $F$ is an almost-star forest and its uncovered vertex set is independent, since the vertices of $I$ are isolated in $G-V(S)$.

It remains to consider $H\cong K_2$, with edge $uv$. Since $G$ is connected, we may assume that $u$ has a neighbor in $S$. If $uc\in E(G)$, add $u$ to $S$ as a leaf; the uncovered set $I\cup\{v\}$ is independent. Otherwise, we may label the leaves so that $ux_1\in E(G)$. If $t\ge3$, take $F$ to consist of the two vertex-disjoint stars with edge sets
$\{uv,ux_1\}$ and $\{cx_j\colon\, 2\le j\le t\}$. Both stars are large, and the uncovered set is $I$.

Finally, suppose that $t=2$, and put $y=x_2$. If $y$ has a neighbor $w\in I$, take the two vertex-disjoint stars with edge sets $\{uv,ux_1\}$ and $\{yc,yw\}$. The uncovered set is $I\setminus\{w\}$. Otherwise, take $F$ to be the path $vux_1c$. The uncovered set $I\cup\{y\}$ is independent because $y$ has no neighbor in $I$. The final assertion follows from Lemma~\ref{a-lem-cover}. \qed
\end{proof}

\subsection{Conflict graphs and the linear bound}\label{a-sec-conflicts}

In this subsection, $G$ is a multigraph. Let $F$ be a subgraph of $G$. For each $e\in E(G)$, let $\widehat e$ be a formal copy of $e$. The \emph{$F$-conflict graph} $\Gamma_F$ is the simple graph with vertex set $\{\widehat e:e\in E(G)\}$ in which distinct vertices $\widehat e_1$ and $\widehat e_2$ are adjacent if and only if some edge $h=xy\in E(F)$ satisfies, after possibly interchanging $x$ and $y$,
\[
 V_G(e_1)\cap V_G(h)=\{x\}
 \quad\text{and}\quad
 V_G(e_2)\cap V_G(h)=\{y\}.
\]
Each such edge $h$ is an \emph{$F$-witness} for $\widehat e_1\widehat e_2$. Parallel edges of $G$ correspond to distinct vertices of $\Gamma_F$.

Fix a proper edge-coloring $\psi:E(G)\to[k]$. For each $i\in[k]$, let $M_i=\{e\in E(G):\psi(e)=i\}$ and define
\begin{equation}\label{a-eq-mono-conflict}
 B_i=\Gamma_F[\{\widehat e:e\in M_i\}].
\end{equation}
Since each $M_i$ is a matching, two distinct edges in $M_i$ are adjacent to a common edge of $F$ if and only if their copies are adjacent in $B_i$. The following example shows that $B_i$ need not be bipartite.

\begin{example}\label{a-ex-triangle}
Let $F$ be the path $x_0x_1x_2x_3$, and write $h_j=x_jx_{j+1}$ for $0\le j\le2$. Obtain $G$ from $F$ by adding two new vertices $y_1,y_2$ and the edges $e_0=x_0x_3$, $e_1=x_1y_1$, and $e_2=x_2y_2$. The subgraph $F$ is total edge-dominating in $G$. Give $G$ the proper edge-coloring in which $e_0,e_1,e_2$ receive color $1$, $h_0,h_2$ receive color $2$, and $h_1$ receives color $3$.

The edges $h_0,h_1,h_2$ are $F$-witnesses for
$\widehat e_0\widehat e_1$,
$\widehat e_1\widehat e_2$, and
$\widehat e_2\widehat e_0$, respectively.
Hence $B_1\cong K_3$, so a monochromatic conflict graph may require three colors even when $F$ is a path; see
Fig.~\ref{a-fig-conflict}.
\end{example}

% Requires \usepackage{tikz}
\definecolor{conflictOne}{RGB}{0,114,178}
\definecolor{conflictTwo}{RGB}{213,94,0}
\definecolor{conflictThree}{RGB}{0,158,115}

% Requires \usepackage{tikz}
\definecolor{conflictOne}{RGB}{0,114,178}
\definecolor{conflictTwo}{RGB}{213,94,0}
\definecolor{conflictThree}{RGB}{0,158,115}

\begin{figure}[tb]
\centering
\begin{minipage}[t]{.46\textwidth}
\centering
\begin{tikzpicture}[
    scale=.95,
    vertex/.style={
        circle,
        draw=black,
        fill=black,
        minimum size=5pt,
        inner sep=0pt
    }
]
\node[vertex,label=above left:{$x_0$}]
    (x0) at (-1.55,1.3) {};
\node[vertex,label=above:{$x_1$}]
    (x1) at (-.55,.35) {};
\node[vertex,label=above:{$x_2$}]
    (x2) at (.55,.35) {};
\node[vertex,label=above right:{$x_3$}]
    (x3) at (1.55,1.3) {};
\node[vertex,label=below:{$y_1$}]
    (y1) at (-.55,-.75) {};
\node[vertex,label=below:{$y_2$}]
    (y2) at (.55,-.75) {};

% Edges of F
\draw[conflictTwo,very thick]
    (x0)--node[left,text=black]{$h_0$}(x1);
\draw[conflictThree,very thick]
    (x1)--node[below,text=black]{$h_1$}(x2);
\draw[conflictTwo,very thick]
    (x2)--node[right,text=black]{$h_2$}(x3);

% Edges of color 1
\draw[conflictOne,line width=.7pt]
    (x0)--node[above,text=black]{$e_0$}(x3);
\draw[conflictOne,line width=.7pt]
    (x1)--node[left,text=black]{$e_1$}(y1);
\draw[conflictOne,line width=.7pt]
    (x2)--node[right,text=black]{$e_2$}(y2);
\end{tikzpicture}
\par\small (a) The graph $G$; thick edges form $F$.
\end{minipage}\hfill
\begin{minipage}[t]{.50\textwidth}
\centering
\begin{tikzpicture}[
    scale=.95,
    vertex/.style={
        circle,
        draw=black,
        fill=black,
        minimum size=5pt,
        inner sep=0pt
    }
]
\node[vertex,fill=conflictOne,
    label=above:{$\widehat e_1$}]
    (e1) at (-1.2,1.3) {};
\node[vertex,fill=conflictOne,
    label=above:{$\widehat e_2$}]
    (e2) at (1.2,1.3) {};
\node[vertex,fill=conflictOne,
    label=right:{$\widehat e_0$}]
    (e0) at (0,.55) {};
\node[vertex,fill=conflictThree,
    label=right:{$\widehat h_1$}]
    (h1) at (0,-.25) {};
\node[vertex,fill=conflictTwo,
    label=below:{$\widehat h_2$}]
    (h2) at (-1.2,-.75) {};
\node[vertex,fill=conflictTwo,
    label=below:{$\widehat h_0$}]
    (h0) at (1.2,-.75) {};

\draw (e1)--(e2)--(e0)--(e1);
\draw (e0)--(h1);
\draw (e1)--(h2)--(h0)--(e2);
\end{tikzpicture}
\par\small (b) The full conflict graph $\Gamma_F$.
\end{minipage}

\caption{A proper edge-coloring of $G$ with
$F=x_0x_1x_2x_3$ and $B_1\cong K_3$.}
\label{a-fig-conflict}
\end{figure}

The following lemma provides the main counting argument.

\begin{lemma}\label{a-lem-degenerate}
Let $G$ be a multigraph, let $F\subseteq G$ be an almost-star forest, and let $\psi$ be a proper edge-coloring of $G$. Then each graph $B_i$ defined in~\eqref{a-eq-mono-conflict} is $2$-degenerate and hence $3$-colorable.
\end{lemma}

\begin{proof}
Fix $i$ and write $B=B_i$. Suppose that $B$ has a nonempty subgraph $H$ with $\delta(H)\ge3$. Put $n=|V(H)|$ and $M=\{e\in M_i\colon\, \widehat e\in V(H)\}$. Since $M$ is a matching, its edges have $2n$ distinct endpoints, and each edge of $F$ witnesses at most one edge of $H$.

Let $P$ be the $P_4$-component of $F$, or the empty graph if no such component exists. Let $E_P$ be the set of edges of $H$ having a witness in $P$, and choose one such witness for each edge of $E_P$. Put $E^*=E(H)\setminus E_P$, $p=|E_P|$, and $m=|E^*|$. Let $U=V(P)\cap V_G(M)$ and $t=|U|$. The chosen witnesses for $E_P$ are distinct edges of $P[U]$. Since $P[U]$ is a forest on $t$ vertices,
\begin{equation}\label{a-eq-path-count}
p\le |E(P[U])|\le
\begin{cases}
0, & \text{if }t=0,\\
t-1, & \text{if }t\ge1.
\end{cases}
\end{equation}
For each edge of $E^*$, choose a witness in a star component of $F$. Let $\mathcal S$ be the set of star components containing a chosen witness, and put $s=|\mathcal S|$.

Suppose that an edge $e\in M$ meets neither the center of a star in $\mathcal S$ nor an internal vertex of $P$. Each endpoint of $e$ then lies on at most one chosen witness: a star leaf or an endpoint of $P$ lies on only one edge of $F$, whereas a vertex outside $V(F)$ or the center of a star outside $\mathcal S$ lies on no chosen witness. Thus $d_H(\widehat e)\le2$, a contradiction. Every edge of $M$ therefore meets a center of a star in $\mathcal S$ or an internal vertex of $P$. Since $M$ is a matching and $P$ has at most two internal vertices,
\begin{equation}\label{a-eq-centers}
s\ge n-2.
\end{equation}
If $t=0$, no edge of $M$ meets $P$, so the same argument gives
\begin{equation}\label{a-eq-centers-no-path}
s\ge n.
\end{equation}

The $m$ chosen star witnesses are distinct and have distinct leaf endpoints, since distinct edges of a simple star have distinct leaves and the components of $F$ are vertex-disjoint. These $m$ leaves, the $s$ star centers, and the $t$ vertices
of $U$ are pairwise distinct endpoints of edges in $M$. Consequently,
\begin{equation}\label{a-eq-endpoint-count}
t+s+m\le2n.
\end{equation}
If $t=0$, equations~\eqref{a-eq-centers-no-path} and~\eqref{a-eq-endpoint-count} give $|E(H)|=m\le n$. If $t\ge1$, equations~\eqref{a-eq-path-count}, \eqref{a-eq-centers}, and~\eqref{a-eq-endpoint-count} give
\[
|E(H)|=m+p
\le 2n-s-t+(t-1)
=2n-s-1
\le n+1.
\]
However, $H$ is simple and $\delta(H)\ge3$, so $n\ge4$ and $|E(H)|\ge3n/2>n+1$, a contradiction. Thus $B$ is $2$-degenerate. \qed
\end{proof}

\noindent \textbf{Proof of Theorem~\ref{a-main-general}.}
The assertion is immediate if $G$ is edgeless. Fix a proper $k$-edge-coloring $\psi:E(G)\to[k]$. We color the edges of each component of $G$ using the common palette $[k]\times[3]$. On each component of order at most two, set $\varphi(e)=(\psi(e),1)$. Every edge in such a component is rich, and every non-isolated edge has a rich neighbor. Now let $C$ be a component of order at least three. By Theorem~\ref{a-thm-almost-star}, $C$ contains a total edge-dominating almost-star forest $F$. Form the conflict graphs $B_i$ using $F$ and the restriction of $\psi$ to $E(C)$. By Lemma~\ref{a-lem-degenerate}, each $B_i$ admits a proper coloring $\sigma_i:V(B_i)\to[3]$. For $e\in E(C)$, define
\(
\varphi(e)=\bigl(\psi(e),\sigma_{\psi(e)}(\widehat e)\bigr).
\)
The first coordinate ensures that $\varphi$ is proper and refines $\psi$.

To show that every edge $h\in E(F)$ is rich, take distinct $e,e'\in N_C(h)$. If $\psi(e)\ne\psi(e')$, then $\varphi(e)\ne\varphi(e')$. Otherwise, let $\psi(e)=\psi(e')=i$. Properness of $\psi$ implies that $e$ and $e'$ are vertex-disjoint. Since both are adjacent to $h$, they meet opposite endpoints of $h$. Thus $h$ witnesses the edge $\widehat e\widehat e'$ of $B_i$, so
$\sigma_i(\widehat e)\ne\sigma_i(\widehat e')$. Hence $\varphi$ is injective on $N_C(h)$, and $h$ is rich. Since $F$ is total edge-dominating, every edge of $C$ has a rich neighbor. Combining the colorings of the components gives the required
refinement with at most $3k$ colors. Taking $k=\chi'(G)$ yields $\rn(G)\le3\chi'(G)$. \qed

%If every monochromatic conflict graph $B_i$ is bipartite, the same argument yields a refinement with at most $2k$ colors. Thus a suitable choice of the witness forest and the initial edge-coloring may improve the factor three to two.

\section{Exact values and optimal colorings}\label{a-sec-exact}
Throughout this section, all graphs are simple.

\subsection{Complete bipartite graphs}\label{a-sec-bipartite}

\noindent \textbf{Proof of Theorem~\ref{a-main-bipartite}.}
By symmetry, assume $m\le n$. If $m=1$, a proper edge-coloring assigns distinct colors to all $n$ edges. Every edge is rich and, unless $n=1$, has a rich neighbor. Hence $\rn(K_{1,n})=n$. Now assume $m\ge2$. Every rich-neighbor edge-coloring of $K_{m,n}$ contains a rich edge, so Lemma~\ref{a-local-lower} gives
\begin{equation}\label{a-eq-bipartite-lower}
\rn(K_{m,n})\ge m+n-1.
\end{equation}

If $m=n=2$, then $K_{m,n}\cong C_4$. A proper edge-coloring using exactly three colors has cyclic color sequence $1,2,3,2$, up to rotation and relabeling. Its only rich edges are two opposite edges, neither of which has a rich neighbor. Assigning four distinct colors makes every edge rich, so $\rn(K_{2,2})=4$.

Suppose $2\le m<n$, and let $A,B$ be the parts, with $|A|=m$ and $|B|=n$. Choose $b_0\in B$ and assign the colors $1,\ldots,m$ to its incident edges. The remaining graph is $K_{m,n-1}$, with maximum degree $n-1$. By K\H{o}nig's line-coloring theorem~\cite{Konig}, it has a proper edge-coloring using $\{m+1,\ldots,m+n-1\}$. Every edge $ab_0$ is rich: its neighbors at $b_0$ and at $a$ receive distinct colors from disjoint palettes. Since $m\ge2$, every edge has a rich neighbor among the edges incident with $b_0$.

Finally, suppose $m=n=t\ge3$, with parts $A=\{a_0,\ldots,a_{t-1}\}$ and $B=\{b_0,\ldots,b_{t-1}\}$. Let $S=[t]$ and $T=\{t+1,\ldots,2t-1\}$, and set
\[
\varphi(a_0b_j)=j+1\quad(0\le j<t),\qquad
\varphi(a_ib_0)=i+1\quad(1\le i<t).
\]
By K\H{o}nig's theorem, color the remaining $K_{t-1,t-1}$ properly with $T$. The resulting coloring is proper, with $\Phi(a_0)=\Phi(b_0)=S$ and $\Phi(a_i)=\Phi(b_i)=\{i+1\}\cup T$ for $1\le i<t$. By~\eqref{a-palette-rich}, every edge $a_0b_j$ with $j\ge1$ and every edge $a_ib_0$ with $i\ge1$ is rich. Each of these edges has another such edge at the same
center, since $t-1\ge2$. Every remaining edge is adjacent to one of these rich edges. Thus the coloring is rich-neighbor and uses $2t-1$ colors, attaining~\eqref{a-eq-bipartite-lower}. \qed

\subsection{Complete graphs: equality and structure}\label{c-equality}

We first establish the lower bound $\rn(K_n)\ge2n-3$ for $n\ge3$ and attain it when $n$ is odd. We then study the subgraph formed by the rich edges in colorings attaining this bound, identify the obstructions at small orders, and give further constructions using one-factorizations. A \emph{one-factor} is a perfect matching, and a \emph{one-factorization} is a partition of the edge set into one-factors.

\begin{lemma}\label{c-one-star}
For $n\ge3$, we have $\rn(K_n)\ge2n-3$ and
\(
\rn(K_n)\le
\begin{cases}
2n-3, & \text{if $n$ is odd},\\
2n-2, & \text{if $n$ is even}.
\end{cases}
\)
\end{lemma}

\begin{proof}
The lower bound follows from Lemma~\ref{a-local-lower}, since $\delta(K_n)=n-1$.

For the upper bounds, fix $v\in V(K_n)$. Assign distinct colors to the $n-1$ edges incident with $v$, and properly edge-color $K_n-v$ using a disjoint palette of $\chi'(K_{n-1})$ colors. Every edge incident with $v$ is rich. Since $n\ge3$, each of these edges is adjacent to another rich edge. Every remaining edge is also adjacent to a rich edge incident with $v$. Thus we obtain a rich-neighbor edge-coloring with $n-1+\chi'(K_{n-1})$ colors.

The stated bounds follow from $\chi'(K_{n-1})=n-2$ when $n$ is odd and $\chi'(K_{n-1})=n-1$ when $n$ is even. \qed
\end{proof}

\begin{lemma}\label{c-structure}
Let $n\ge4$, and let $\varphi:E(K_n)\to\mathcal C$ be a proper edge-coloring with $|\mathcal C|=2n-3$. Let $R$ be the spanning subgraph formed by all rich edges. Then $R$ is a star forest. If $\varphi$ is a rich-neighbor edge-coloring, then $R$ has no isolated edges and at most one isolated vertex.
\end{lemma}

\begin{proof}
For every rich edge $uv$, equation~\eqref{a-palette-rich} and $|\Phi(u)|=|\Phi(v)|=n-1$ give
\begin{equation}\label{c-complement} 
\Phi(u)\cup\Phi(v)=\mathcal C,
\qquad
\Phi(u)\cap\Phi(v)=\{\varphi(uv)\}.
\end{equation}
Suppose that $uv$ and $vw$ are distinct edges of $R$. Put $\alpha=\varphi(uv)$, $\beta=\varphi(vw)$, and $M=\mathcal C\setminus\Phi(v)$. By~\eqref{c-complement},
\(
\Phi(u)=M\cup\{\alpha\}, \Phi(w)=M\cup\{\beta\}.
\)
Since $\alpha\ne\beta$, we have $\Phi(u)\cap\Phi(w)=M$, whose size is $n-2\ge2$. Thus $uw$ is not rich, so $R$ contains no triangle. Suppose now that $wz\in E(R)$ for some $z\notin\{u,v,w\}$, and put $\gamma=\varphi(wz)$. Properness gives
$\gamma\in\Phi(w)\setminus\{\beta\}=M$, and
\eqref{c-complement} yields
\(
\Phi(z)=(\Phi(v)\setminus\{\beta\})\cup\{\gamma\}.
\)
Hence $\Phi(u)\cap\Phi(z)=\{\alpha,\gamma\}$. The color of $uz$ must belong to this intersection, but $\alpha$ is already used on $uv$ and $\gamma$ on $wz$, contradicting properness. Thus $R$ contains neither a triangle nor a path of length three, so every nontrivial component of $R$ is a star.

If $\varphi$ is a rich-neighbor edge-coloring, no component of $R$ consists of a single edge, since that edge would have no rich neighbor. Nor can $R$ have two isolated vertices, since the edge joining them in $K_n$ would have no rich neighbor. \qed
\end{proof}

\begin{lemma}\label{c-component-lower}
Let $n\ge4$ be even, and let $\varphi$ be a rich-neighbor edge-coloring of $K_n$ with $2n-3$ colors. Its rich-edge graph has at least two nontrivial components. If $n\equiv2\pmod4$, then it has at least three nontrivial components.
\end{lemma}

\begin{proof}
Let $\mathcal C$ be the color set, let $\Phi(v)$ be the set of colors on edges incident with $v$, and let $R$ be the rich-edge graph. By Lemma~\ref{c-structure}, every nontrivial component of $R$ is a star with at least two edges, and $R$ has at most one isolated vertex. Also, $R$ contains an edge by the rich-neighbor condition. By~\eqref{c-complement}, if $uv$ is rich, then every edge incident with $v$ other than $uv$ has its color outside $\Phi(u)$.

Suppose that $R$ has exactly one nontrivial component, with center $a$. All but at most one vertex of $K_n-a$ are leaves of this star. Hence every edge of $K_n-a$ has its color in $\mathcal C\setminus\Phi(a)$, which has size $n-2$. Since $n-1$ is odd, each color class in $K_n-a$ has at most $(n-2)/2$ edges. Thus these colors cover at most $(n-2)^2/2<\binom{n-1}{2}$ edges, a contradiction. Therefore, $R$ has at least two nontrivial components.

Now suppose that $n\equiv2\pmod4$ and that $R$ has exactly two nontrivial components. Write $n=2r$, where $r\ge3$ is odd. Let $a,b$ be the centers of the two stars, and let $A,B$ be their leaf sets, respectively. Put $p=|A|$ and $q=|B|$, so $p,q\ge2$. Let $s\in\{0,1\}$ be the number of isolated vertices of $R$. Then $p+q=2r-2-s$. Let $t=|\Phi(a)\cap\Phi(b)|$ and put $T=\mathcal C\setminus(\Phi(a)\cup\Phi(b))$. Since $|\Phi(a)|=|\Phi(b)|=2r-1$, we have $|T|=t-1$. Every edge joining $A$ to $B$ has its color in $T$. The edges from $a$ to $B$ and from $b$ to $A$ have their colors in $\Phi(a)\setminus\Phi(b)$ and $\Phi(b)\setminus\Phi(a)$, respectively. Properness therefore gives $t-1\ge\max\{p,q\}$ and $2r-1-t\ge\max\{p,q\}$. Adding these inequalities yields $\max\{p,q\}\le r-1$. Together with $p+q=2r-2-s$, this implies, after interchanging the stars if necessary, that $q=r-1$, $p=r-1-s$ and $t=r$. Put $Q=\Phi(b)\setminus\Phi(a)$. Then $|Q|=|T|=r-1$ and $\mathcal C\setminus\Phi(a)=T\cup Q$.

At each vertex of $A$, the $r-1$ edges to $B$ use all colors in $T$. Every other incident edge, except the edge to $a$, must therefore have its color in $Q$. If $s=0$, let $J=K_n[A\cup\{b\}]$. If $s=1$, let $x$ be the isolated vertex of $R$ and let $J=K_n[A\cup\{b,x\}]-bx$. In both cases, $J$ has $r$ vertices and $\binom r2-s$ edges, all colored from $Q$.

Since $r$ is odd, each color class in $J$ has at most $(r-1)/2$ edges. Thus the $r-1$ colors in $Q$ cover at most $(r-1)^2/2$ edges. However,
\[
|E(J)|-\frac{(r-1)^2}{2}
=\binom r2-s-\frac{(r-1)^2}{2}
=\frac{p-s}{2}>0,
\]
since $p\ge2$ and $s\le1$. This contradiction proves the second assertion. \qed
\end{proof}

\begin{proposition}\label{c-small}
We have $\rn(K_4)=6$ and $\rn(K_6)=10$.
\end{proposition}

\begin{proof}
By Lemma~\ref{c-one-star}, it suffices to exclude a rich-neighbor edge-coloring of $K_4$ with five colors and of $K_6$ with nine colors.

By Lemmas~\ref{c-structure} and~\ref{c-component-lower}, the rich-edge graph of the first coloring would have at least two nontrivial star components, each containing at least three vertices. This is impossible on four vertices. For the second coloring, the same lemmas require at least three such components, which is impossible on six vertices. \qed
\end{proof}

\begin{proposition}\label{c-two-stars}
If $n\ge8$ and $n\equiv0\pmod4$, then $\rn(K_n)=2n-3$.
\end{proposition}

\begin{proof}
By Lemma~\ref{c-one-star}, it suffices to construct a rich-neighbor edge-coloring of $K_n$ with $2n-3$ colors. Write $n=2r$, where $r\ge4$ is even. Choose distinct vertices $a,b$, and partition the remaining vertices into $A=\{a_1,\ldots,a_{r-1}\}$ and $B=\{b_1,\ldots,b_{r-1}\}$. We construct a coloring in which all edges $aa_i$ and $bb_i$ are rich.

The induced subgraphs $K_n[\{a\}\cup B]$ and $K_n[\{b\}\cup A]$ are copies of $K_r$. Since $r$ is even, each decomposes into $r-1$ perfect matchings. The edges between $A$ and $B$ also decompose into $r-1$ perfect matchings, since they form $K_{r-1,r-1}$. Give each of these $3(r-1)$ matchings a different color. Next, give the $r-1$ matchings $\{aa_i,bb_i\}$, where $1\le i\le r-1$, distinct new colors, and give $ab$ one further new color. Every edge is now colored, and each color class is a matching, so the coloring is proper.

For each $i$, the edges incident with $a$ use only colors assigned to $K_n[\{a\}\cup B]$ and the new colors. The edges incident with $a_i$ other than $aa_i$ lie in $K_n[\{b\}\cup A]$ or join $A$ to $B$. Their colors are therefore different from all colors used at $a$, so $aa_i$ is rich. Similarly, every $bb_i$ is rich. These rich edges form two stars that together cover all
vertices, each with at least two edges. Hence every edge has a distinct adjacent rich edge. The coloring is therefore rich-neighbor and uses $3(r-1)+(r-1)+1=2n-3$ colors. \qed
\end{proof}

\subsection{Matching constructions for complete graphs}\label{c-constructions}
It remains to consider $n\ge10$ with $n\equiv2\pmod4$. We give three constructions according to $n$ modulo 12.
Each partitions $E(K_n)$ into $2n-3$ matchings and makes every edge of three prescribed vertex-disjoint stars rich. A matching is \emph{near-perfect} if it leaves exactly one vertex uncovered, and it is \emph{rainbow} if its edges
have pairwise distinct colors.

For odd $q\ge1$, write $V(K_{q+1})=\{w,v_0,\ldots,v_{q-1}\}$. Using $\mathbb Z_q$ as the color set, define
\[
\varphi(v_iv_j)=i+j\quad(0\le i<j<q),\qquad
\varphi(wv_i)=2i\quad(0\le i<q),
\]
where all colors are taken modulo $q$. Each vertex is incident with exactly one edge of each color, so the color classes form a one-factorization. Deleting $w$ yields $q$ near-perfect matchings of $K_q$. For any positive integer $r$, let the parts of $K_{r,r}$ be $X=\{x_0,\ldots,x_{r-1}\}$ and $Y=\{y_0,\ldots,y_{r-1}\}$. For $0\le t\le r-1$, put $M_t=\{x_i y_{i+t}\colon\,0\le i \le r-1\}$, where subscripts are taken modulo $r$. Then $M_0,\ldots,M_{r-1}$ form a one-factorization of $K_{r,r}$.

\begin{lemma}\label{c-balanced-matching}
Let $k=2h\ge2$ be an even integer. Then there exist a proper edge-coloring $\varphi:E(K_k)\to[k]$ and a partition $V(K_k)=X_0\cup X_1$, with $|X_0|=|X_1|=h$, such that $h$ color classes are perfect matchings and each remaining color class is a matching of size $h-1$ that leaves one vertex in each of $X_0$ and $X_1$ uncovered.
\end{lemma}

\begin{proof}
A one-factorization of $K_k$, together with an empty color class, gives a proper edge-coloring from $[k]$. Among all proper edge-colorings $\varphi:E(K_k)\to[k]$, choose one minimizing $\sum_{i=1}^k |E_i|^2$, where $E_i=\varphi^{-1}(i)$. If $|E_i|\ge|E_j|+2$, then the subgraph with edge set $E_i\cup E_j$ has an alternating path component with one more edge of $E_i$ than of $E_j$. Interchanging the two colors on this path preserves properness and decreases the sum of squares, a contradiction. Thus the class sizes differ by at most one. Their average is $(k-1)/2=h-\tfrac12$, so exactly $h$ classes have size $h$ and the other $h$ have size $h-1$.

Each vertex is incident with $k-1$ distinct colors and hence is missed by exactly one color class. The pairs of vertices missed by the classes of size $h-1$ therefore partition $V(K_k)$. Construct the required partition by placing one vertex of each pair in $X_0$ and the other in $X_1$. \qed
\end{proof}

\begin{lemma}\label{c-switched-bipartite}
Let $A,B$ be vertex sets with $A\cap B=\varnothing$ and $|A|=|B|=r$, and let $u,v\notin A\cup B$ with $u\ne v$. Choose $A_1\subseteq A$ and $B_1\subseteq B$ with $|A_1|=|B_1|$, and put $A_0=A\setminus A_1$ and $B_0=B\setminus B_1$. Let $H$ be the graph with vertex set $V(H)=A\cup B\cup\{u,v\}$ and edge set
\(E(H)=E(A,B)\cup\{uv\}\cup E(u,A_1\cup B_0) \cup E(v,A_0\cup B_1).\)
Then $E(H)$ can be partitioned into $r+1$ perfect matchings.
\end{lemma}

\begin{proof}
Take a one-factorization of $K_{A\cup\{u\},\,B\cup\{v\}}$. Each factor not containing $uv$ contains edges $va$ and $ub$ for some $a\in A$ and $b\in B$. These pairs define a bijection $f:A\to B$. Since $|A_1|=|B_1|$, applying a suitable permutation of $B$ to the entire factorization allows us to assume that $f(A_1)=B_1$. For each $a\in A_1$, replace $va$ and $uf(a)$ in their common factor by $ua$ and $vf(a)$. Each factor remains a perfect matching, and the resulting $r+1$ factors partition $E(H)$.
\qed
\end{proof}

For $k\ge6$, the next lemma also follows from the rainbow perfect matching theorem of Woolbright and Fu~\cite{WoolbrightFu}. We give a direct construction valid for all even $k\ge4$.

\begin{lemma}\label{c-rainbow}
Let $k=2h\ge4$ be an even integer, and let $x,y\in V(K_k)$ with $x\ne y$. Then $K_k$ has a one-factorization admitting a rainbow matching of size $h-1$ that misses exactly $x$ and $y$.
\end{lemma}

\begin{proof}
Use the cyclic one-factorization above with $q=2h-1$, $w=x$, and $v_{2h-2}=y$. The matching
\(
M=\{v_{2j}v_{2j+1}\colon\,0\le j\le h-2\}
\)
misses exactly $x$ and $y$. Its edge colors are $4j+1$ modulo $2h-1$, which are pairwise distinct since $\gcd(4,2h-1)=1$.
Thus $M$ is rainbow. \qed
\end{proof}

The next lemma gives a common criterion for the three constructions: at each leaf, the colors other than its star-edge color are precisely those missing at its center.

\begin{lemma}\label{c-three-palettes}
Let $G$ be a complete graph with vertex partition $\{a\},\{b\},\{c\},\{x\},A,B,C$, where $|A|,|B|,|C|\ge2$. Let $\varphi$ be a proper edge-coloring of $G$ with color set $\mathcal C$. Suppose that
\[
\begin{aligned}
\Phi(t)\setminus\{\varphi(at)\}
  &=\mathcal C\setminus\Phi(a) &&(t\in A),\\
\Phi(t)\setminus\{\varphi(bt)\}
  &=\mathcal C\setminus\Phi(b) &&(t\in B),\\
\Phi(t)\setminus\{\varphi(ct)\}
  &=\mathcal C\setminus\Phi(c) &&(t\in C).
\end{aligned}
\]
Then every edge in $E(a,A)\cup E(b,B)\cup E(c,C)$ is rich, and $\varphi$ is a rich-neighbor edge-coloring.
\end{lemma}

\begin{proof}
For $t\in A$, the first identity gives $\Phi(a)\cap\Phi(t)=\{\varphi(at)\}$, so $at$ is rich by~\eqref{a-palette-rich}. The same argument applies to the other two stars. Each star has at least two edges, and their union covers all vertices except $x$. Thus their union is total edge-dominating by Lemma~\ref{a-lem-cover}, and $\varphi$ is a rich-neighbor edge-coloring. \qed
\end{proof}

In each construction, the color set is the disjoint union of $\mathcal P_{AB}$, $\mathcal P_{AC}$, $\mathcal P_{BC}$, $\mathcal S_A$, $\mathcal S_B$, $\mathcal S_C$, and $\mathcal O$. Here $\mathcal P_{AB}$ consists of the colors missing at $a$ and $b$ but used at $c$, and $\mathcal S_A$ consists of those missing at $a$ but used at $b$ and $c$. The other indexed sets have analogous meanings, while $\mathcal O$ consists of the colors used at all three centers. We realize the palettes in Table~\ref{c-tab-palettes}. The last two entries in each row are complementary in the full color set, so Lemma~\ref{c-three-palettes} applies.

\begin{table}[htbp]
\centering
\small
\begin{tabular}{@{}ccc@{}}
\toprule
Star edges & Center palette
& \shortstack[c]{Leaf palette excluding the star-edge color}\\
\midrule
$E(a,A)$
& $\mathcal P_{BC}\cup\mathcal S_B
   \cup\mathcal S_C\cup\mathcal O$
& $\mathcal P_{AB}\cup\mathcal P_{AC}\cup\mathcal S_A$\\
$E(b,B)$
& $\mathcal P_{AC}\cup\mathcal S_A
   \cup\mathcal S_C\cup\mathcal O$
& $\mathcal P_{AB}\cup\mathcal P_{BC}\cup\mathcal S_B$\\
$E(c,C)$
& $\mathcal P_{AB}\cup\mathcal S_A
   \cup\mathcal S_B\cup\mathcal O$
& $\mathcal P_{AC}\cup\mathcal P_{BC}\cup\mathcal S_C$\\
\bottomrule
\end{tabular}
\caption{Complementary palettes for the three prescribed stars.}
\label{c-tab-palettes}
\end{table}

Each construction uses the vertex partition in Lemma~\ref{c-three-palettes} and the color sets in Table~\ref{c-tab-palettes}, with $\mathcal O=\{\alpha,\beta\}$ and set sizes given in Table~\ref{c-tab-families}. We assign a distinct color to each matching. Whenever edges are replaced, the old edges are removed from their color classes. After all replacements, the color classes partition $E(K_n)$.

Since $n\equiv2\pmod4$, its residue modulo $12$ is $2$, $6$, or $10$. Constructions I, II, and III handle these cases, respectively. In each construction, we use $k=|C|$ as the parameter. The congruence restrictions in Table~\ref{c-tab-families}
select the remaining orders; Constructions I and III work for all even $k$ in their stated ranges.

\begin{table}[htbp]
\centering
\small
\setlength{\tabcolsep}{4pt}
\begin{tabular}{@{}cccc@{}}
\toprule
 & Construction I & Construction II & Construction III\\
\midrule
Parameter range
& $k\ge4$
& $k\ge6$
& $k\ge2$\\
$k\bmod4$
& $0$
& $2$
& $2$\\
Order $n$
& $3k+2$
& $3k$
& $3k+4$\\
Leaf-set sizes
& $(k-1,k-1,k)$
& $(k-2,k-2,k)$
& $(k,k,k)$\\
$\mathcal P$-set sizes
& $(k+1,k,k)$
& $(k-1,k,k)$
& $(k+1,k+1,k+1)$\\
$\mathcal S$-set sizes
& $(k-1,k-1,k)$
& $(k-1,k-1,k-2)$
& $(k,k,k)$\\
\bottomrule
\end{tabular}
\caption{Parameters of the three constructions.
Leaf-set and $\mathcal S$-set sizes are listed in the order $A,B,C$, and $\mathcal P$-set sizes in the order $AB,AC,BC$. In every construction, $|\mathcal O|=2$.}
\label{c-tab-families}
\end{table}

\paragraph{Construction I: two odd leaf sets.}

\begin{lemma}\label{c-odd-leaves}
For every even integer $k\ge4$, $\rn(K_{3k+2})\le6k+1$.
\end{lemma}

\begin{proof}
Write $k=2h$. Use the vertex partition in Lemma~\ref{c-three-palettes} with $|A|=|B|=k-1$ and $|C|=k$. Choose disjoint color sets with the sizes listed for Construction I in Table~\ref{c-tab-families}, and put $\mathcal O=\{\alpha,\beta\}$. These sets contain $6k+1$ colors in total.

\emph{The color sets $\mathcal P_{AC}$,
$\mathcal P_{BC}$, and $\mathcal S_C$.} By Lemma~\ref{c-balanced-matching}, color $K_C$ with $\mathcal S_C$ and choose a partition $C=C_0\cup C_1$, where $|C_0|=|C_1|=h$, such that $h$ color classes are perfect matchings and each of the other $h$ classes has $h-1$ edges and misses one vertex in each part. The missed pairs partition $C$. Choose $A'\subseteq A$ and $B'\subseteq B$ with $|A'|=|B'|=h$.

Use $\mathcal P_{AC}$ for a one-factorization of $K_{C,A\cup\{b\}}$. For each $v\in C_0$, replace $bv$ in its color class by $xv$ and $bw$, choosing the vertices $w\in B'$ bijectively. The replaced edges lie in distinct classes, and each replacement preserves the matching property. The resulting classes cover every vertex of $A\cup C\cup\{b\}$ and partition
\(
E(A,C)\cup E(b,C_1)\cup E(x,C_0)\cup E(b,B').
\)
Similarly, use $\mathcal P_{BC}$ for a one-factorization of $K_{C,B\cup\{a\}}$. For each $v\in C_1$, replace $av$ by $xv$ and a distinct edge of $E(a,A')$. The resulting classes cover every vertex of $B\cup C\cup\{a\}$ and partition
\[
E(B,C)\cup E(a,C_0)\cup E(x,C_1)\cup E(a,A').
\]

To each $\mathcal S_C$-class missing $v_0\in C_0$ and $v_1\in C_1$, add $bv_0$ and $av_1$. Add $ab$ to one of the original perfect classes. To each of the other $h-1$ perfect classes, add one edge from $E(a,A\setminus A')$ and one from $E(b,B\setminus B')$, using every such edge once. This is possible since $|A\setminus A'|=|B\setminus B'|=h-1$. Each class remains a matching and now covers every vertex of $C\cup\{a,b\}$.

\emph{The color sets $\mathcal S_A$ and $\mathcal S_B$.}
Choose $a_*\in A$ and $b_*\in B$. By Lemma~\ref{c-rainbow}, choose one-factorizations of $K_{A\cup\{b\}}$ and $K_{B\cup\{a\}}$ with palettes $\mathcal S_A$ and $\mathcal S_B$, respectively, admitting rainbow matchings $M_A$ and $M_B$ of size $h-1$ that miss exactly $\{b,a_*\}$ and $\{a,b_*\}$. Choose partitions
\(
A=U_A\cup V_A\cup\{a_*\},
B=U_B\cup V_B\cup\{b_*\},
\)
where $|U_A|=|V_A|=|U_B|=|V_B|=h-1$, such that each edge of $M_A$ joins $U_A$ to $V_A$ and each edge of $M_B$ joins $U_B$ to $V_B$.

For each $uv\in M_A$, with $u\in U_A$ and $v\in V_A$, replace $uv$ in its $\mathcal S_A$-class by $cu$ and $xv$. Since $M_A$ is rainbow, each class is modified at most once and remains a matching. There are $(k-1)-(h-1)=h$ unmodified classes. Add the edges of $E(c,C_0)$ to these classes, one edge per class. Every $\mathcal S_A$-class now covers every vertex of $A\cup\{b,c\}$. Apply the analogous replacements to $M_B$, and add the edges of $E(c,C_1)$ to the $h$ unmodified $\mathcal S_B$-classes, one edge per class. Every $\mathcal S_B$-class then covers every vertex of $B\cup\{a,c\}$. Reserve the removed edges of $M_A$ and $M_B$ for $\mathcal P_{AB}$.

\emph{The remaining color classes.}
Apply Lemma~\ref{c-switched-bipartite} with $(u,v)=(c,x)$, $A_1=V_A$, and $B_1=U_B$. Use $k$ colors of $\mathcal P_{AB}$ for the resulting one-factorization of the graph with edge set
\[
\begin{aligned}
E(A,B)\cup\{cx\}
&\cup E(c,V_A\cup(B\setminus U_B))\cup E(x,(A\setminus V_A)\cup U_B).
\end{aligned}
\]
Assign the remaining color of $\mathcal P_{AB}$ to
\(
M_A\cup M_B\cup\{ca_*,xb_*\}.
\)
This is a perfect matching on $A\cup B\cup\{c,x\}$ and contains no edge of the preceding graph. Thus every vertex of $A\cup B\cup\{c,x\}$ sees every color in $\mathcal P_{AB}$. Finally, set
\(
\varphi(ac)=\varphi(bx)=\alpha,
\varphi(bc)=\varphi(ax)=\beta.
\)

\emph{Verification.}
Every color class is a matching. The edges of $M_A$ and $M_B$ have been removed from their original classes and assigned to $\mathcal P_{AB}$. The assignments at $c$ and $x$ partition their edges to $A$ as follows:
\[
\begin{aligned}
E(c,A)
 &=E(c,V_A\cup\{a_*\})\cup E(c,U_A),\\
E(x,A)
 &=E(x,U_A\cup\{a_*\})\cup E(x,V_A),
\end{aligned}
\]
where the first set in each row uses $\mathcal P_{AB}$ and the second uses $\mathcal S_A$. The corresponding partitions for $B$ use $\mathcal P_{AB}$ and $\mathcal S_B$. The edges of $E(c,C_0)$ and $E(c,C_1)$ use $\mathcal S_A$ and $\mathcal S_B$, respectively. Among the six edges on $\{a,b,c,x\}$, the edge $ab$ uses $\mathcal S_C$, the edge $cx$ uses $\mathcal P_{AB}$, and the other four use $\mathcal O$. Together with the first step, these assignments partition $E(K_{3k+2})$.

Each center has the palette in Table~\ref{c-tab-palettes}. Every vertex of $A$ sees all colors in $\mathcal P_{AB}\cup\mathcal P_{AC}\cup\mathcal S_A$ on edges other than its star edge. These $3k$ colors form its entire non-star palette. The same argument applies to $B$ and $C$, so all palette equalities in Table~\ref{c-tab-palettes} hold. Lemma~\ref{c-three-palettes} completes the proof. \qed
\end{proof}

\paragraph{Construction II: two smaller even leaf sets.}

\begin{lemma}\label{c-even-leaves}
For every integer $k\ge6$ with $k\equiv2\pmod4$, $\rn(K_{3k})\le6k-3$.
\end{lemma}

\begin{proof}
Write $k=2h$, where $h\ge3$ is odd. Use the vertex partition in Lemma~\ref{c-three-palettes} with $|A|=|B|=k-2$ and $|C|=k$. Choose disjoint color sets with the sizes listed for Construction II in Table~\ref{c-tab-families}, and put $\mathcal O=\{\alpha,\beta\}$. These sets contain $6k-3$ colors in total. Partition $A=A_0\cup A_1$ and $B=B_0\cup B_1$, where each part has size $h-1$. Write $C=\{v_0,\ldots,v_{k-1}\}$, and put $C_0=\{v_{2j}:0\le j<h\}$ and $C_1=\{v_{2j+1}:0\le j<h\}$, with subscripts on the vertices $v_i$ taken modulo $k$.

\emph{The color sets $\mathcal P_{AB}$,
$\mathcal S_A$, and $\mathcal S_B$.} Apply Lemma~\ref{c-switched-bipartite} with $(u,v)=(c,x)$ and the subsets $A_1,B_1$. Use $\mathcal P_{AB}$ for the resulting one-factorization of the graph with edge set
\begin{equation}\label{c-c-pab}
E(A,B)\cup E(c,A_1\cup B_0)
\cup E(x,A_0\cup B_1)\cup\{cx\}.
\end{equation}

Take a one-factorization of $K_{A\cup\{b,w\}}$, where $w$ is an auxiliary vertex, and assign its $k-1$ factors the colors in $\mathcal S_A$. In each factor, replace the edge $wu$ by $cu$ if $u\in A_0$, by $xu$ if $u\in A_1$, and by $bx$ if $u=b$. Exactly $h-1$ classes now cover $c$. Add the $h$ edges of $E(c,C_0)$ to the remaining classes, one edge per class. Each class remains a matching and covers $A\cup\{b,c\}$. The resulting classes partition $E(K_A)\cup E(b,A)\cup E(c,A_0) \cup E(x,A_1)\cup\{bx\}\cup E(c,C_0)$.

Similarly, use $\mathcal S_B$ for a one-factorization of $K_{B\cup\{a,w\}}$. Replace $wu$ by $cu$ if $u\in B_1$, by $xu$ if $u\in B_0$, and by $ax$ if $u=a$. Add the edges of $E(c,C_1)$ to the $h$ classes missing $c$, one edge per class. Each resulting class is a matching covering $B\cup\{a,c\}$, and these classes partition $E(K_B)\cup E(a,B)\cup E(c,B_1) \cup E(x,B_0)\cup\{ax\}\cup E(c,C_1)$.

\emph{The color sets $\mathcal P_{AC}$, $\mathcal P_{BC}$, and $\mathcal O$.} Choose $u_A\in A$ and $u_B\in B$, and set $\varphi(ac)=\varphi(bu_B)=\alpha$ and $\varphi(bc)=\varphi(au_A)=\beta$. Both color classes are matchings. Order $A\cup\{b,x\}$ as $\ell_0=b,\ell_1=x,\ell_2,\ldots,\ell_{k-1}$. The matchings $M_t=\{\ell_i v_{t+i}:0\le i<k\}$, where $0\le t<k$, form a one-factorization of $K_{A\cup\{b,x\},C}$. Assign them the colors in $\mathcal P_{AC}$. For each even $t$, replace $bv_t,xv_{t+1}$ by $bu_t^B,v_tv_{t+1}$, where the $u_t^B$ are distinct vertices of $B\setminus\{u_B\}$. These choices are possible since $k-3=2h-3\ge h$. Each class remains a matching covering $A\cup C\cup\{b\}$. Their edge sets partition $E(A,C)\cup E(b,C_1)\cup E(x,C_0) \cup\{bu_t^B:0\le t<k,\ t\text{ even}\}\cup D_0$, where $D_0=\{v_{2i}v_{2i+1}:0\le i<h\}$.

For $\mathcal P_{BC}$, use the corresponding cyclic one-factorization of $K_{B\cup\{a,x\},C}$, with $a,x$ as the first two vertices. For each odd $t$, replace $av_t,xv_{t+1}$ by $au_t^A,v_tv_{t+1}$, using distinct vertices $u_t^A\in A\setminus\{u_A\}$. Each class remains a matching covering $B\cup C\cup\{a\}$. Their edge sets partition $E(B,C)\cup E(a,C_0)\cup E(x,C_1) \cup\{au_t^A:0\le t<k,\ t\text{ odd}\}\cup D_1$, where $D_1=\{v_{2i+1}v_{2i+2}:0\le i<h\}$.

\emph{The color set $\mathcal S_C$.} Since $h$ is odd, for each $j\in\{0,1\}$, decompose $K_{C_j}$ into near-perfect matchings $N_0^j,\ldots,N_{h-1}^j$, where $N_i^j$ misses $v_{2i+j}$. Use $h$ colors of $\mathcal S_C$ for the matchings $N_i^0\cup N_i^1\cup\{bv_{2i},av_{2i+1}\}$, where $0\le i<h$. Each covers $C\cup\{a,b\}$.

The remaining uncolored edges within $C$ form $K_{C_0,C_1}-D_0-D_1$. They decompose into the perfect matchings $F_d=\{v_{2i}v_{2(i+d)+1}:0\le i<h\}$, where $1\le d\le h-2$. Indeed, $D_0$ and $D_1$ correspond to cyclic shifts $0$ and $-1$ modulo $h$, respectively, so these are precisely the remaining shifts. Add $ab$ to $F_1$. To each of the other $h-3$ factors, add one unused edge of $E(a,A)$ and one unused edge of $E(b,B)$, using every such edge once. Each of these stars has exactly $(k-2)-1-h=h-3$ unused edges, so these additions are possible. Assign the remaining $h-2$ colors of $\mathcal S_C$ to the augmented factors. When $h=3$, only $F_1$ occurs and no star edges remain. Every $\mathcal S_C$-class is a matching covering $C\cup\{a,b\}$.

\emph{Verification.}
The listed edge sets partition $E(K_{3k})$. The edges inside $C$ are partitioned by $D_0,D_1$ and the restrictions of the $\mathcal S_C$-classes to $C$. The edges from $c,x$ to $A,B$ are partitioned by \eqref{c-c-pab} and the two auxiliary-vertex constructions, while $E(c,C)$ is split between $\mathcal S_A$ and $\mathcal S_B$. The edges of $E(a,A)$ are divided among $\mathcal O,\mathcal P_{BC},\mathcal S_C$ in numbers $1,h,h-3$, respectively. The same counts hold for $E(b,B)$, with $\mathcal P_{AC}$ in place of $\mathcal P_{BC}$. On $\{a,b,c,x\}$, the edge $ab$ uses $\mathcal S_C$, the edges $ac,bc$ use $\mathcal O$, the edges $ax,bx$ use $\mathcal S_B,\mathcal S_A$, respectively, and $cx$ uses $\mathcal P_{AB}$. Every color class is a matching.

The center palettes are those in Table~\ref{c-tab-palettes}. Every leaf receives all colors in its prescribed non-star palette on non-star edges. Each such palette has $3k-2$ colors, equal to the number of incident non-star edges, so the palette equalities in the table hold. Lemma~\ref{c-three-palettes} completes the proof. \qed
\end{proof}

\paragraph{Construction III: three equal leaf sets.}

\begin{lemma}\label{c-equal-leaves}
For every even integer $k\ge2$, $\rn(K_{3k+4})\le6k+5$.
\end{lemma}

\begin{proof}
Write $k=2h$. Use the vertex partition in Lemma~\ref{c-three-palettes} with $|A|=|B|=|C|=k$. Choose disjoint color sets with the sizes listed for Construction III in Table~\ref{c-tab-families}, and put $\mathcal O=\{\alpha,\beta\}$. These sets contain $6k+5$ colors in total.

Apply Lemma~\ref{c-balanced-matching} to $K_A,K_B,K_C$, using $\mathcal S_A,\mathcal S_B,\mathcal S_C$, respectively. Choose the resulting partitions $A=A_0\cup A_1$, $B=B_0\cup B_1$, and $C=C_0\cup C_1$, where each part has size $h$. Each coloring has $h$ perfect classes and $h$ classes of size $h-1$, each missing one vertex in each part. The missed pairs partition the corresponding leaf set.

\emph{The color sets $\mathcal P_{AB}$,
$\mathcal P_{BC}$, and $\mathcal P_{AC}$.} Apply Lemma~\ref{c-switched-bipartite} to the three rows of Table~\ref{c-tab-a-p}, with ordered parts $(A,B)$, $(B,C)$, and $(C,A)$, respectively. Each resulting graph has a one-factorization with
$k+1$ factors. Assign its factors the colors in the indicated color set.

\begin{table}[htbp]
\centering
\small
\setlength{\tabcolsep}{4pt}
\begin{tabular}{@{}cc@{}}
\toprule
Color set & Edge set\\
\midrule
$\mathcal P_{AB}$
& $E(A,B)\cup E(c,A_1\cup B_0)
   \cup E(x,A_0\cup B_1)\cup\{cx\}$\\
$\mathcal P_{BC}$
& $E(B,C)\cup E(a,B_1\cup C_0)
   \cup E(x,B_0\cup C_1)\cup\{ax\}$\\
$\mathcal P_{AC}$
& $E(C,A)\cup E(b,C_1\cup A_0)
   \cup E(x,C_0\cup A_1)\cup\{bx\}$\\
\bottomrule
\end{tabular}
\caption{Edge sets for the three one-factorizations.}
\label{c-tab-a-p}
\end{table}

\emph{The color sets $\mathcal S_A$,
$\mathcal S_B$, and $\mathcal S_C$.} Consider the chosen $\mathcal S_A$-coloring of $K_A$. To each class of size $h-1$ missing $u\in A_0$ and $v\in A_1$, add $cu$ and $bv$. To each of the $h$ perfect classes, add one edge of $E(b,B_0)$ and one edge of $E(c,C_1)$, using every such edge once. Each class remains a matching covering $A\cup\{b,c\}$. Extend the $\mathcal S_B$- and $\mathcal S_C$-classes cyclically as specified in Table~\ref{c-tab-a-s}. They cover $B\cup\{a,c\}$ and $C\cup\{a,b\}$, respectively, in every color.

\begin{table}[htbp]
\centering
\small
\setlength{\tabcolsep}{4pt}
\begin{tabular}{@{}ccc@{}}
\toprule
Color set
& \shortstack{Edges added to classes of size $h-1$}
& \shortstack{Edges added to perfect classes}\\
\midrule
$\mathcal S_A$
& $E(c,A_0),\ E(b,A_1)$
& $E(b,B_0),\ E(c,C_1)$\\
$\mathcal S_B$
& $E(a,B_0),\ E(c,B_1)$
& $E(c,C_0),\ E(a,A_1)$\\
$\mathcal S_C$
& $E(b,C_0),\ E(a,C_1)$
& $E(a,A_0),\ E(b,B_1)$\\
\bottomrule
\end{tabular}
\caption{Extensions of the internal color classes. Each listed edge is added exactly once.}
\label{c-tab-a-s}
\end{table}

Together with the original colorings inside $A,B,C$, the two tables partition all edges except $ab,bc,ca$. At this stage, the center palettes are those in Table~\ref{c-tab-palettes}, with $\mathcal O$ omitted.

\emph{The center triangle.}
Choose a color $\gamma\in\mathcal S_C$ whose original class in $K_C$ is a perfect matching. Its extended class contains $au$ and $bv$ for some $u\in A_0$ and $v\in B_1$. Replace these two edges in the $\gamma$-class by $ab$, and set $\varphi(au)=\varphi(bc)=\beta$ and $\varphi(bv)=\varphi(ac)=\alpha$. The modified $\gamma$-class and the new
$\alpha$- and $\beta$-classes are matchings. Both $a$ and $b$ retain color $\gamma$, and every center receives both $\alpha$ and $\beta$. Thus the center palettes are exactly those in Table~\ref{c-tab-palettes}.

\emph{Verification.}
The construction partitions $E(K_{3k+4})$ into matchings and hence defines a proper edge-coloring. Before the final recoloring, every vertex of $A$ receives all colors in $\mathcal P_{AB}\cup\mathcal P_{AC}\cup\mathcal S_A$ on its non-star edges. These are $3k+2$ colors, equal to the number of incident non-star edges. The same argument applies to $B$ and $C$. Since only two prescribed star edges were recolored, all non-star palettes remain unchanged. The palette equalities in Table~\ref{c-tab-palettes} therefore hold. As each leaf set has size $k\ge2$, Lemma~\ref{c-three-palettes} completes the proof. \qed
\end{proof}

\noindent \textbf{Proof of Theorem~\ref{a-main-complete}.}
Lemma~\ref{c-one-star} gives $\rn(K_n)\ge2n-3$ for $n\ge3$, with equality when $n$ is odd. Proposition~\ref{c-small} settles $n=4,6$, and Proposition~\ref{c-two-stars} gives equality when $n\equiv0\pmod4$ and $n\ge8$.

It remains to consider $n\equiv2\pmod4$ with $n\ge10$. If $n\equiv2\pmod{12}$, apply Lemma~\ref{c-odd-leaves} with $k=(n-2)/3$. If $n\equiv6\pmod{12}$, apply Lemma~\ref{c-even-leaves} with $k=n/3$. If $n\equiv10\pmod{12}$, apply Lemma~\ref{c-equal-leaves} with $k=(n-4)/3$. In each case, the parameter satisfies the hypotheses of the corresponding lemma, and the construction uses $2n-3$ colors. This proves equality in all remaining cases. \qed

\noindent \textbf{Proof of Theorem~\ref{a-main-structure}.} Let $n\ge7$. By Theorem~\ref{a-main-complete}, every optimal rich-neighbor edge-coloring of $K_n$ uses $2n-3$ colors. Lemma~\ref{c-structure} therefore gives the asserted structure of its full rich-edge graph. Every such coloring contains a rich edge, so $s(n)\ge1$. Lemma~\ref{c-component-lower} gives $s(n)\ge2$ for even $n$ and $s(n)\ge3$ when $n\equiv2\pmod4$.

For the upper bounds, Lemma~\ref{c-one-star} gives a coloring with a spanning rich star when $n$ is odd. Proposition~\ref{c-two-stars} gives a coloring with two vertex-disjoint rich stars covering all vertices when $n\equiv0\pmod4$. When $n\equiv2\pmod4$, Lemmas~\ref{c-odd-leaves}, \ref{c-even-leaves}, and~\ref{c-equal-leaves} give colorings with three vertex-disjoint rich stars covering all but one vertex. All these colorings use $2n-3$ colors and hence are optimal.

In each case, every nontrivial component of the full rich-edge graph meets a prescribed star and therefore contains that star, since at most one vertex is uncovered. Thus the numbers of nontrivial components are at most one, two, and three, respectively. Together with the lower bounds, this proves the stated values of $s(n)$. \qed

\section{Concluding remarks}\label{a-sec-questions}

Theorem~\ref{a-main-general} gives $\rn(G)\le3\chi'(G)$ for every  multigraph $G$. We conjecture that the factor three can be reduced to two.

\begin{conjecture}\label{a-factor-two}
Every multigraph $G$ satisfies $\rn(G)\le2\chi'(G)$.
\end{conjecture}

The coefficient two would be best possible, even asymptotically, since Theorem~\ref{a-main-bipartite} gives $\rn(K_{\Delta,\Delta})/\chi'(K_{\Delta,\Delta})=2-1/\Delta$ for $\Delta\ge3$. Equality in the conjectured bound holds for $K_{2,2}$, $K_4$, and $K_6$. Theorems~\ref{a-main-bipartite} and~\ref{a-main-complete} verify the conjecture for complete bipartite graphs and complete graphs.

A possible approach is to choose, for each connected multigraph $G$ of order at least three, a total edge-dominating almost-star forest $F$ and a proper $\chi'(G)$-edge-coloring $\psi$ such that every monochromatic conflict graph $B_i$ is bipartite. The refinement argument in the proof of Theorem~\ref{a-main-general} would then use at most two subcolors per original color.

\section*{Data availability}
No datasets were generated or analyzed in this study.

\section*{Declaration of generative AI and AI-assisted technologies}

GPT-6 Astra Ultra was used to assist with language editing, manuscript organization, and preparing TikZ code. The author reviewed the AI-assisted content, revised it as needed, and takes full responsibility for the final manuscript.


\begin{thebibliography}{99}

\bibitem{BoyerGoddard}
G. Boyer and W. Goddard, Disjoint isolating sets and graphs with maximum isolation number, \emph{Discrete Appl. Math.} \textbf{356} (2024), 110--116.


\bibitem{ChenNanZhou}
L. Chen, C. Nan and X. Zhou, On the rich-neighbor edge-coloring of sparse graphs, \emph{Discrete Appl. Math.} \textbf{378} (2026), 538--546.


\bibitem{ChenTanZhou}
L. Chen, Y. Tan and X. Zhou, On the rich-neighbor edge-colorings of subcubic graphs, \emph{Graphs Combin.} \textbf{41} (2025) 89.


\bibitem{FouquetJolivet}
J.-L. Fouquet and J.-L. Jolivet, Strong edge-colorings of graphs and applications to multi-$k$-gons, \emph{Ars Combin.} \textbf{16A} (1983), 141--150.


\bibitem{Harary}
F. Harary, On the notion of balance of a signed graph, \emph{Michigan Math. J.} \textbf{2} (1953/54), 143--146.


\bibitem{Konig}
D. K\H{o}nig, \emph{Theorie der endlichen und unendlichen Graphen: Kombinatorische Topologie der Streckenkomplexe},
Akademische Verlagsgesellschaft, Leipzig, 1936.


\bibitem{PetrusevskiSkrekovski}
M. Petru\v{s}evski and R. \v{S}krekovski, Proper edge-colorings with a rich neighbor requirement, \emph{Discrete Math.} \textbf{347} (2024) 113803.


\bibitem{Shannon}
C. E. Shannon, A theorem on coloring the lines of a network, \emph{J. Math. Phys.} \textbf{28} (1949), 148--152.


\bibitem{Vizing}
V. G. Vizing, On an estimate of the chromatic class of a $p$-graph, \emph{Diskret. Analiz}, \textbf{3} (1964), 25--30 (in Russian).


\bibitem{WoolbrightFu}
D. E. Woolbright and H.-L. Fu, On the existence of rainbows in $1$-factorizations of $K_{2n}$, \emph{J. Combin. Des.} \textbf{6} (1998), 1--20.


\bibitem{Zaslavsky}
T. Zaslavsky, Signed graphs, \emph{Discrete Appl. Math.} \textbf{4} (1982), 47--74.

\end{thebibliography}
\end{document}